\documentclass[11pt]{article}

\usepackage[margin=1in]{geometry}
\usepackage{amsmath,amssymb,amsthm,mathtools}
\usepackage[colorlinks=true,linkcolor=blue,urlcolor=blue,citecolor=blue]{hyperref}
\usepackage{natbib}

\newtheorem{thm}{Theorem}[section]
\newtheorem{lem}[thm]{Lemma}
\newtheorem{cor}[thm]{Corollary}
\newtheorem{remark}[thm]{Remark}

\numberwithin{equation}{section}

\newcommand{\LR}{\mathrm{LR}}

\newcommand{\dd}{\,\mathrm{d}}

\newcommand{\unif}{\mathrm{unif}}

\newcommand{\Pois}{\mathrm{Poi}}
\newcommand{\Multi}{\mathrm{Mult}}

\newcommand{\Bin}{\mathrm{Bin}}

\newcommand{\Scal}{\mathcal{S}}
\newcommand{\Prp}[1]{\Pr\left[#1 \right]}
\newcommand{\reals}{\mathbb R}

\newcommand{\Ncal}{\mathcal{N}}

\newcommand{\one}{\mathbf{1}}

\newcommand{\Var}[1]{\mathrm{Var}\left[ #1\right]}
\newcommand{\ex}[1]{\ensuremath{\mathbb{E}\left[ #1\right]}}
\newcommand{\exsub}[2]{\ensuremath{\mathbb{E}_{#1}\left[ #2\right]}}

\newcommand{\Cov}[1]{\mathrm{Cov} \left[#1\right]}
\newcommand{\Corr}[1]{\mathrm{corr} \left[#1\right]}

\title{Sharp Lower Bound on the Minimax Risk for Multinomial Uniformity Testing via a Conditional Central Limit Theorem}

\author{Alon Kipnis\\
School of Computer Science, Reichman University\\
\href{mailto:alon.kipnis@runi.ac.il}{alon.kipnis@runi.ac.il}}

\date{}

\begin{document}

\maketitle

\begin{abstract}
We study minimax goodness-of-fit testing for uniformity from $n$ multinomial observations over $N$ categories against $\ell_p$ departures of size $\epsilon_n$. Writing $u_n:=\epsilon_n^2 n\,N^{3/2-2/p}/\sqrt{2}$ for the associated signal-to-noise ratio, we focus on the intermediate regime $N=o(n^2)$ with $u_n\to u^*\in(0,\infty)$, in which the minimax risk converges to a nontrivial constant. In the Poissonized version of the problem this constant equals $2\Phi(-u^*/2)$ \cite{Kipnis2025minimax}, yielding an upper bound on the multinomial minimax risk. Here we prove the matching lower bound. The key step is a conditional central limit theorem for weighted sums under a Poisson mixture prior, conditioned on the total count. Together with the upper bound in \cite{Kipnis2025minimax}, this gives an exact sharp-constant characterization of the multinomial minimax risk in the intermediate regime.
\end{abstract}



\section{Problem setup}
\label{sec:setup}

Denote by $O_i$ the count of category $i$. For a given $N$-dimensional multinomial distribution $\mathbf{q} := (q_1,\ldots,q_N)$ with $q_i \geq 0$ and $\sum_{i=1}^N q_i = 1$, let $H(\mathbf{q})$ denote the model
\begin{align}
    \label{eq:hyp_Q}
H(\mathbf{q}) \,:\, (O_1,\ldots,O_N) \sim \Multi(n,\mathbf{q}).
\end{align}
Namely, $n$ independent samples from $\mathbf{q}$. 
The problem of testing the uniformity of $\mathbf{q}$ corresponds to testing the null hypothesis 
\begin{align}
    H_0 = H(\mathbf{q}_{\unif}),\quad \text{where} \quad \mathbf{q}_{\unif} := (1/N,\ldots,1/N), 
    \label{eq:null_hyp}
\end{align} 
against alternatives in which the distribution $\mathbf{q}$ deviates from uniformity in the $\ell_p$ sense (cf.\ \cite{ingster2003nonparametric,chhor2022sharp}):
\[
H_1 = H(\mathbf{q}),\quad \text{for some~~} \mathbf{q} \in A_N(\epsilon,p),
\]
where 
\begin{align}
    \label{eq:alt_hyp_def}
A_N(\epsilon,p) := \left\{ \mathbf{q} \in [0,1]^N \, :\, \left\| \mathbf{q} - \mathbf{q}_{\unif} \right\|_p \geq \epsilon, \, \|\mathbf{q}\|_1 = 1 \right\}.
\end{align}
Notice that $A_N(\epsilon,p)$ is not empty provided $\epsilon \leq N^{1/p-1}$. 

A test $\psi$ maps a random sample $(O_1,\ldots,O_N)$ to a decision $\{0,1\}$. The risk of $\psi$ for testing $H_1$ against $H_0$ is defined as
\[
R(\psi; H_1, H_0) := \Prp{\psi = 1 \mid H(\mathbf{q}_{\unif})}  + \sup_{\mathbf{q} \in A_N(\epsilon,p)} \Prp{\psi = 0 \mid H(\mathbf{q})}. 
\]
The minimax risk is defined as
\begin{align}
\label{eq:minimax_def}
    R^* := \inf_{\psi} R(\psi; H_1, H_0).
\end{align}

\section{Main results}
\label{sec:main-results}

Assume that $n$ tends to infinity, with $N$ and $\epsilon$ depending on $n$. 
Define
\begin{align}
    \label{eq:SNR_def}
    u_n := \frac{\epsilon_n^2 n N_n^{3/2-2/p}}{\sqrt{2}}.
\end{align}
We call $u_n$ the signal-to-noise ratio (SNR).

It follows from previous works that $R^* \to 0$ if and only if $u_n \to \infty$, and 
$R^* \to 1$ if and only if $u_n \to 0$, provided $A_{N}(\epsilon,p)$ is non-empty \cite{balakrishnan2019hypothesis,chhor2022sharp}, with the case $p=1$ famously characterized in \cite{paninski2008coincidence}. Furthermore, by the Poisson setting studied in \cite{Kipnis2025minimax}, we have
\begin{align}
\limsup_{n \to \infty} R^* \leq 2 \Phi(-u^*/2),
\label{eq:upper_bound}
\end{align}
whenever $u_n \to u^*$ with $N = o(n^2)$ for some $u^* \in (0,\infty)$. The main result of this note establishes the matching lower bound.

\begin{thm}
\label{thm:main}
    Consider the hypothesis testing problem under the minimax setting of \eqref{eq:minimax_def}. Suppose that $N_n$ goes to infinity with $N_n = o(n^2)$, $n = O(N)$ and $\lim_{n \to \infty} u_n = u^*$, for some $u^* \in (0,\infty)$. Then
    \begin{align*}
        \liminf_{n \to \infty} R^* \geq 2 \Phi(-u^*/2).  
    \end{align*}
\end{thm}

By the upper bound \eqref{eq:upper_bound}, we get:
\begin{cor} 
\label{cor:main}
Under the assumptions of Theorem~\ref{thm:main},
\begin{align}
    \label{eq:minimax_risk_equality}
    \lim_{n \to \infty} R^* = 2 \Phi(-u^*/2).  
\end{align}
\end{cor}

For the proof of Theorem~\ref{thm:main}, we analyze the likelihood ratio test associated with a least-favorable prior. For a real sequence $w := (w_0,w_1,\ldots)$, define the statistic
\begin{align}
    T(w) := \sum_{i=1}^N w_{O_i},
    \label{eq:test_statistic}
\end{align}
and the weights
\begin{align}
    \label{eq:wm_star_equivalent}
    w_m^* \propto (m-n/N)^2 - m,\quad m=0,1,\ldots,
\end{align}
which arise as the large-$N$ approximation to the likelihood ratio weights (Lemma~\ref{lem:unconditioned_LR_test} and Lemma~\ref{lem:weight_approximation}).

A key component in the proof of Theorem~\ref{thm:main} is the following conditional central limit theorem for the statistic $T(w^*)$ under a Poisson mixture sampling model.
\begin{thm}
    \label{thm:CLT}
        For each $n$, let $N=N_n$ and let
        $\tilde O_1,\ldots,\tilde O_N$ be independent and identically distributed according to a Poisson mixture law
        \begin{align}
            \label{eq:Poisson_mixture_model}
        Q_i\sim \pi_1,\qquad
        \tilde O_i\mid Q_i\sim \Pois(nQ_i),
        \qquad i=1,\ldots,N,
        \end{align}
        with all pairs $(Q_i, \tilde O_i)$ independent. Let $w =(w_0,w_1,\ldots)$ be a real sequence and define
        \[
        \tilde T_N(w)=\sum_{i=1}^N w_{\tilde O_i},
        \qquad
        S_N=\sum_{i=1}^N \tilde{O}_i.
        \]
        Let $U_N(w)$ have the law of $\tilde T_N(w)$ conditioned on the event
        \begin{align}
        \label{eq:Scal_def}
        \Scal_n := \{S_N=n\}    
        \end{align}
        Assume that, as $n\to\infty$, $N = N_n \to \infty$, $N=o(n^2)$ and the following. 
        \begin{enumerate}    
        \item[(i)] $\Corr{\tilde O_1,w_{\tilde O_1}}=o(1)$.
        
        \item[(ii)] For $Q_1 \sim \pi_1$, $\ex{Q_1}=\frac{1}{N}$ and 
        there is a universal constant $0<C<\infty$ such that the support of $\pi_1$ is in the interval $\left[0,\frac{C}{n}\right]$.
        
        \item[(iii)] We have $\Var{w_{\tilde{O}_1}}>0$ and
        $\ex{|w_{\tilde{O}_1}|^3}<\infty$ for every $n$, and the Lyapunov condition:
        \begin{align*}
        \frac{\ex{\left|w_{\tilde O_1}
        -\ex{w_{\tilde O_1}}\right|^3}}
        {{\sqrt N}\Var{w_{\tilde O_1}}^{3/2}} = o(1).
        \end{align*}
        \end{enumerate}
        Then
        \begin{align*}
        \sup_{x\in\reals} \left| \Prp{ \frac{U_N(w)-\ex{\tilde T_N(w)}}
        {\sqrt{\Var{\tilde T_N(w)}}} \le x } - \Phi(x) \right| \longrightarrow 0.
        \end{align*}
\end{thm}

\section{Preliminaries}
\label{sec:preliminaries}
We first provide several definitions and preliminary results to facilitate the proof of the main results in Section~\ref{sec:proofs} below.

\subsection{Poissonized and Poisson mixture models}
The associated Poissonized problem \eqref{eq:hyp_pois}
corresponds to the case where each category's count obeys a Poisson distribution with parameter $n q_i$, independently for $i=1,\ldots,N$ \cite{morris1975central,jacquet1998analytical,aldous2013probability}. 

For a vector $\mathbf{q} = (q_1,\ldots,q_N)$ with non-negative entries, consider the Poisson sampling probability law $\tilde{H}(\mathbf{q})$ of \eqref{eq:hyp_pois}. 
We also define the following set of $N$-length rate sequences 
\begin{align*}
    \tilde{A}_N(\epsilon,p) := \left\{ \tilde{\mathbf{q}} \in {[0,\infty)}^N \, :\, \left\| \tilde{\mathbf{q}} - \mathbf{q}_{\unif} \right\|_p \geq \epsilon \right\}.
\end{align*}
Apart from the unit sum constraint, this set is identical to $A_N(\epsilon,p)$ of \eqref{eq:alt_hyp_def}. For future use, we define the null Poisson mixture model as
\begin{align*}
    \tilde{H}_0 := \tilde{H}(\mathbf{q}_{\unif}).
\end{align*}

The minimax risk under standard Poisson sampling is defined analogously to $R^*$ of \eqref{eq:minimax_def} by 
\begin{align}
\label{eq:minimax_tilde_def}
\tilde{R}(\tilde{A}_N(\epsilon,p)) := \inf_{\psi} \left( \Prp{\psi=1 \mid \tilde{H}_0} + \sup_{q \in \tilde{A}_N(\epsilon,p)} \Prp{\psi = 0 \mid \tilde{H}(q)} \right).
\end{align}
We additionally define the minimax risk under Poisson sampling \emph{conditioned} on the event $\Scal_n := \{S_N=n\}$ by 
\begin{align}
    \label{eq:minimax_tilde_conditioned_def}
& \tilde{R}^*\mid \Scal_n := \tilde{R}^*(\tilde{A}_N(\epsilon,p) \mid \Scal_n) \\
& := \inf_{\psi} \left( \Prp{\psi=1 \mid \tilde{H}(\mathbf{q}_{\unif} \mid \Scal_n)} + \sup_{\tilde{\mathbf{q}} \in \tilde{A}_N(\epsilon,p)} \Prp{\psi = 0 \mid \tilde{H}(\tilde{\mathbf{q}} \mid \Scal_n)} \right). \nonumber
\end{align}

\subsection{Reduction to a Bayesian setting}
For a prior $\pi$, denote by $\tilde{H}(\pi)$ the Poisson mixture probability model 
\begin{align*}
    \tilde{H}(\pi) \,:\,Q_i \sim \pi_i, \quad \tilde{O}_i \sim \Pois(n Q_i),\quad \forall i=1,\ldots,N,\,\mbox{\rmfamily independently}. 
\end{align*}
The Bayes risk under $\pi \in \Pi_{N}$ and a test $\psi = \psi(o_1,\ldots,o_N)$ in the Poisson sampling setting is defined as
\[
\tilde{\rho}(\pi; \psi) := \Prp{\psi = 1 \mid \tilde{H}(\mathbf{q}_{\unif})} + \Prp{\psi = 0 \mid \tilde{H}(\pi)}.
\]
Denote by $\tilde{H}(\mathbf{q} \mid \Scal_n)$ the law $\tilde{H}(\mathbf{q})$ conditioned on the event $\{S_N=n\}$. We define the conditioned Bayes risk $\tilde{\rho}(\pi; \psi \mid \Scal_n)$ by
\begin{align}
    \label{eq:rho_tilde_pi_cond_def}
    \tilde{\rho}(\pi; \psi \mid \Scal_n) :=  
    \Prp{\psi = 1 \mid \tilde{H}(\mathbf{q}_{\unif} \mid \Scal_n)} + \Prp{\psi = 0 \mid \tilde{H}(\pi \mid \Scal_n)}. 
\end{align}
Notice that $\tilde{H}(\pi \mid \Scal_n)$ is not a multinomial law because of the randomness in $\mathbf{Q}$. We also define 
\[
\tilde{\rho}^*(\pi \mid \Scal_n) := \inf_{\psi} \tilde{\rho}(\pi; \psi \mid \Scal_n).
\]

Define a set of $N$-dimensional product priors with support in $[0,1]^N$ that satisfy a moment condition associated with the $\epsilon$ separation:
\begin{align}
    \Pi_N := \left\{ \pi= \prod_{i=1}^N \pi_i \,:\, \exsub{\mathbf{Q} \sim \pi}{\left\|\mathbf{Q} - \mathbf{q}_{\unif} \right\|_p^p } \geq \epsilon^p,\, \pi_i([0,1])=1 \right\}.
\end{align}
The following lemma states that for $\pi \in \Pi_N$, the minimax risk is bounded from below by the Bayes risk. The proof, provided in Appendix~\ref{sec:proof_bayes_equiv}, uses the law of large numbers applied to a sequence of independent random variables $Q_i \sim \pi_i$.
\begin{lem}
\label{lem:Bayes_equivalence}
    Fix $\pi \in \Pi_N$. As $N \to \infty$, 
    \[
        \tilde{R}^* \mid \Scal_n \geq \tilde{\rho}^*(\pi \mid \Scal_n) + o(1)
    \]
    uniformly in $\epsilon$. 
\end{lem}

Finally, we define a class of tests $\psi_{w,\tau}$ using sequences $(w) = (w_0,w_1,\ldots)$, test statistic $T(w)$ as in \eqref{eq:test_statistic}, and a threshold $\tau \in \reals$, by rejecting $H_0$ when 
\begin{align}
\label{eq:linear_test}
    T(w) \geq \tau.
\end{align}

The notation above follows a general convention that we adopt henceforth: random variables and expressions associated with data sampled from the Poisson distribution are decorated with a tilde; $\rho$ denotes Bayes risk; and expressions involving an optimization, such as an infimum or supremum, are decorated with an asterisk.

\subsection{Least favorable prior candidate}
For the lower bound on the minimax risk, we put a prior on the alternative and bound the minimax risk below by the corresponding Bayes risk. The prior used here is the product prior
\begin{align}
    \label{eq:least_favorable_prior}
    \pi^* = \prod_{i=1}^N \pi_i^*, \quad \text{where} \quad \pi_i^* = \frac{1}{2}\delta_{\frac{1}{N} - \epsilon N^{-1/p}} + \frac{1}{2}\delta_{\frac{1}{N} + \epsilon N^{-1/p}},
\end{align}
where $\delta_x$ is a point mass at $x$. Namely, the prior consists of local two-sided perturbations of the uniform distribution. We then analyze the likelihood ratio test under the mixed multinomial model
\begin{align}
\label{eq:mixed_multinomial_model}
Q \sim \pi^*,\qquad 
(O_1,\ldots,O_N) \mid Q \sim \Multi(n,Q/\|Q\|_1),
\end{align}
whose Bayes risk bounds the asymptotic minimax risk $R^*$ from below.

\subsection{Conditioning and likelihood ratios}

The following lemma is a well-known equivalence between the laws of conditioned Poisson distribution and the multinomial distribution (cf.\ \cite{aldous2013probability,barbour1992poisson}).
\begin{lem}
\label{lem:dist_equivalence}
    For every $\tilde{\mathbf{q}} \in [0,\infty)^N \setminus \{(0,\ldots,0)\}$, $\tilde{H}(\tilde{\mathbf{q}} \mid \Scal_n) = H(\tilde{\mathbf{q}}/\|\tilde{\mathbf{q}}\|_1)$.
\end{lem}

The following lemma states that the equivalence represented in Lemma~\ref{lem:dist_equivalence} holds for the conditional minimax risk defined in \eqref{eq:minimax_tilde_conditioned_def}.
\begin{lem}
\label{lem:risk_equivalence}
    For any $\epsilon$, $p$, $N$, and $n$, 
    \[
    R^* = \tilde{R}^*\mid \Scal_n.
    \]
\end{lem}
\begin{proof}
By Lemma~\ref{lem:dist_equivalence}, we have
    \[
    \Prp{\psi = 1 \mid \tilde{H}(\mathbf{q}_{\unif} \mid \Scal_n)} = \Prp{\psi = 1 \mid H(\mathbf{q}_{\unif})}.
    \]
    Additionally, $A_N(\epsilon,p) \subset \tilde{A}_N(\epsilon,p)$ by definition. Hence, from the definitions of the minimax risk in \eqref{eq:minimax_def} and the conditioned minimax risk in \eqref{eq:minimax_tilde_conditioned_def}, it is enough to show that for any fixed test $\psi$, 
    \begin{align}
        \label{eq:lem:risk_equivalence_to_show}
    \sup_{\tilde{\mathbf{q}} \in \tilde{A}_N(\epsilon,p)} \Prp{\psi = 0 \mid \tilde{H}(\tilde{\mathbf{q}})} \leq \sup_{\mathbf{q} \in A_N(\epsilon,p)} \Prp{\psi = 0 \mid H(\mathbf{q})}.
    \end{align}
    Let $\tilde{\mathbf{q}}^{(k)} \in \tilde{A}_N(\epsilon,p)$ be a sequence converging to the supremum on the left-hand side of \eqref{eq:lem:risk_equivalence_to_show}. 
    By Lemma~\ref{lem:dist_equivalence}, 
    \begin{align*}
        & \sup_{\tilde{\mathbf{q}} \in \tilde{A}_N(\epsilon,p)} \Prp{\psi = 0 \mid \tilde{H}(\tilde{\mathbf{q}})} = \lim_{k \to \infty} \Prp{\psi = 0 \mid \tilde{H}(\tilde{\mathbf{q}}^{(k)})} \\
        & \qquad = \lim_{k \to \infty} \Prp{\psi = 0 \mid H(\tilde{\mathbf{q}}^{(k)}/\|\tilde{\mathbf{q}}^{(k)}\|_1)} \leq \sup_{\mathbf{q} \in A_N(\epsilon,p)} \Prp{\psi = 0 \mid H(\mathbf{q})},
    \end{align*}
where the last transition is because $\frac{\tilde{\mathbf{q}}^{(k)}}{\|\tilde{\mathbf{q}}^{(k)}\|_1} \in A_N(\epsilon,p)$.
\end{proof}

The following lemma connects the likelihood ratios in the Poisson law conditioned on $\Scal_n$ and the unconditioned law. The proof is given in Appendix~\ref{sec:proof_LR_equivalence}.
\begin{lem}
\label{lem:LR_equivalence}
For $\epsilon N^{-1/p} \leq 1/N$, consider the distribution $\pi^*$ of \eqref{eq:least_favorable_prior}. Denote by $\tilde{L}$ the likelihood ratio of $\tilde{H}(\pi^*)$ to $\tilde{H}(\mathbf{q}_{\unif})$. Denote by $\tilde{L}\mid \Scal_n$ the likelihood ratio of 
    $\tilde{H}(\pi^* \mid \Scal_n)$ to $\tilde{H}(\mathbf{q}_{\unif} \mid \Scal_n)$. Let $n,N \to \infty$ with $u_n \to u^* \in (0,\infty)$. For every $(o_1,\ldots,o_N) \in \mathbb{N}^N$ with $\sum_{i=1}^N o_i = n$, we have
    \begin{align*}
    \tilde{L} (o_1,\ldots,o_N \mid \Scal_n) = \tilde{L}(o_1,\ldots,o_N)(1 + O(1/\sqrt{N})).
    \end{align*}
\end{lem}

The following lemma derives the likelihood ratio test in the Poisson setting, showing that it is of the form \eqref{eq:linear_test}. 
\begin{lem}
    \label{lem:unconditioned_LR_test}
    Consider testing the null hypothesis $H_0 = \tilde{H}(\mathbf{q}_{\unif})$ against the alternative $H_1 = \tilde{H}(\pi^*)$, with $\pi^*$ of \eqref{eq:least_favorable_prior}. The likelihood ratio test is of the form $\psi_{w^{\LR},\tau}$ with:
\begin{align}
\label{eq:LR_test_weights}
w^{\LR}_m = \log(g_{m,n/N})\!\left( \epsilon N^{1-1/p} \right),\quad m=0,1,\ldots,
\end{align}
where 
\begin{align}
\label{eq:g_def}
g_{m,\lambda}(s) := \frac{1}{2}e^{-\lambda s} \left( 1 + s \right)^{m} + \frac{1}{2}e^{\lambda s} \left( 1 - s \right)^{m}.
\end{align}
\end{lem}

\begin{proof}
The log-likelihood ratio statistic is given by
\begin{align*}
\ell(\tilde{O}_1,\ldots,\tilde{O}_N;\pi^*) & := \sum_{i=1}^N \log(\tilde{L}_i(\tilde{O}_i;\pi_i^*)), 
\end{align*}
where 
\begin{align*}
\tilde{L}_i(x;\pi_i^*) & = \frac{\Prp{\tilde{O}_i=x \mid \tilde{H}(\pi^*)} }{\Prp{\tilde{O}_i = x \mid H_0} }.
\end{align*}
Note that,
\begin{align*}
\tilde{L}_i(x;\pi_i^*) & = \frac{1}{2}e^{-n \epsilon N^{-1/p}} \left( 1 + \epsilon N^{1-1/p} \right)^{x} + \frac{1}{2}e^{n \epsilon N^{-1/p}} \left( 1 - \epsilon N^{1-1/p} \right)^{x} \\
& = g_{x,n/N}(\epsilon N^{1-1/p}), 
\end{align*}
where the last equality uses $\lambda s = (n/N)\,\epsilon N^{1-1/p} = n \epsilon N^{-1/p}$ with $\lambda = n/N$ and $s = \epsilon N^{1-1/p}$. Set $w^{\LR}_m := \log g_{m,n/N}(\epsilon N^{1-1/p})$, $m=0,1,\ldots$. It follows that the likelihood ratio test rejects for large values of $\tilde{T}(w^{\LR})$. 
\end{proof}

Instead of working with the unconditioned likelihood ratio sequence $w^{\LR}$ of Lemma~\ref{lem:unconditioned_LR_test}, it is simpler to consider its scaled large $N$ approximation: $w^*$ of \eqref{eq:wm_star_equivalent}. The following lemma shows that this approximation does not change the asymptotic distribution of the standardized $\tilde{T}(w^{\LR})$.
\begin{lem} 
\label{lem:weight_approximation}
    Let $w_m^* = (m-(n/N))^2 - m$. Suppose that $n$ and $N$ tend to infinity with $N = o(n^2)$ and $u_n \to u^* \in (0,\infty)$. Then, 
    \begin{align} 
        \lim_{n \to \infty} \left[ \frac{\tilde{T}(w^{\LR})-\ex{\tilde{T}(w^{\LR}) \mid \tilde{H}(\pi^*)}}{\sqrt{\Var{\tilde{T}(w^{\LR}) \mid \tilde{H}(\pi^*)}}} - \frac{\tilde{T}(w^*)-\ex{\tilde{T}(w^*) \mid \tilde{H}(\pi^*)}}{\sqrt{\Var{\tilde{T}(w^*) \mid \tilde{H}(\pi^*)}}}  \right] = 0, 
    \end{align}
    in distribution. 
\end{lem}

\begin{proof}
The condition $N = o(n^2)$ and $u_n \to u^*$, implies $s := \epsilon N^{1-1/p} = o(1)$. By a Taylor expansion of $g_{m,\lambda}(s)$ (Lemma~\ref{lem:taylor_expansion} in the Appendix), we have
\begin{align*}
    \log g_{m,\lambda}(s) = \frac{s^2}{2} \left((m-\lambda)^2 - m \right) + o\!\left(s^2(e^m + \lambda^2)\right), 
\end{align*}
so, with $\lambda = n/N$ and $s = \epsilon N^{1-1/p} = o(1)$, we get
\begin{align}
\label{eq:gm_Taylor}
w^{\LR}_m = \frac{\epsilon^2 N^{2-2/p}}{2} \left(w_m^* + o(e^m + \lambda^2)\right).
\end{align}
Because \eqref{eq:gm_Taylor} shows that $w^{\LR}_m$ grows at most exponentially in $m$, and because $\pi^*$ has a finite support and all moments of the Poisson distribution exist, all moments of $\tilde{T}(w^{\LR})$ exist. The same is clearly true for $\tilde{T}(w^*)$. Thus, the difference between the moments of the two standardized statistics is $o(1)$, so the claim follows from standard moment convergence theorem (cf.\ \cite{anderson2010introduction}). 
\end{proof}

The following lemma applies Theorem~\ref{thm:CLT} to show that the Bayes risk of the likelihood ratio test in the conditioned Poisson setting is asymptotically identical to the Bayes risk of this test in the unconditioned setting. The proof is in Appendix~\ref{sec:proof_conditional_limit_theorem}.
\begin{lem}
    \label{lem:conditional_limit_theorem}
    Consider a test of the form \eqref{eq:linear_test} with
    $w_m^* = (m-(n/N))^2 - m$, 
    \[
    \tau^* = \frac{1}{2} n^2 \epsilon^2 N^{1-2/p},
    \]
    and $\pi^*$ of \eqref{eq:least_favorable_prior}. As $n,N \to \infty$ with $N = N_n = o(n^2)$ and $u_n \to u^*$, we have
    \[
    \tilde{\rho}(\pi^*; \psi_{w^*,\tau^*} \mid \Scal_n) + o(1) = \tilde{\rho}(\pi^*; \psi_{w^*,\tau^*}) = 2\Phi(-u^*/2).
    \]
\end{lem}

\section{Proofs of main results}
\label{sec:proofs}

\subsection{Proof of Theorem~\ref{thm:main}}

\begin{proof}
By Lemma~\ref{lem:risk_equivalence}, 
\[
R^* = \tilde{R}^* \mid \Scal_n.
\]
Fix the prior $\pi^*$ defined in \eqref{eq:least_favorable_prior}. This prior satisfies $\pi^* \in \Pi_N$, hence Lemma~\ref{lem:Bayes_equivalence} implies that
\[
\tilde{R}^*\mid \Scal_n \geq \tilde{\rho}^*(\pi^* \mid \Scal_n) + o(1).
\]

Let $\psi^{\LR}$ denote the likelihood ratio test of Lemma~\ref{lem:unconditioned_LR_test}. By Neyman-Pearson theory, 
\[
\tilde{\rho}^*( \pi^* \mid \Scal_n) = \tilde{\rho}^* ( \pi^* ; \psi^{\LR} \mid \Scal_n).
\]
Let $\psi^* := \psi_{w^*,\tau^*}$ be the test with weights $w_m^* = (m-n/N)^2 - m$, $m=0,1,\ldots$, and threshold $\tau^* := \frac{1}{2} n^2 \epsilon^2 N^{1-2/p}$. By Lemma~\ref{lem:weight_approximation}, $\psi^*$ approximates the likelihood ratio test in the sense that 
\[
\tilde{\rho}^* ( \pi^* ; \psi^{\LR} \mid \Scal_n) = \tilde{\rho}(\pi^*;\psi^* \mid \Scal_n) + o(1).
\]
By Lemma~\ref{lem:conditional_limit_theorem}, as $N\to\infty$ with $u_n \to u^*$,
\[
\tilde{\rho}(\pi^*; \psi^* \mid \Scal_n) + o(1) = 2\Phi(-u^*/2).
\]
Combining the displays and taking $\liminf$ yields
\[
\liminf_{N\to\infty} R^* \geq 2\Phi(-u^*/2),
\]
as claimed.
\end{proof}

\subsection{Proof of Theorem~\ref{thm:CLT}}
\label{sec:proof_CLT}
\begin{proof}
The proof uses the characteristic function of the joint distribution of $(\tilde{T}_N(w), S_N)$. The main step is an inversion formula for the discrete random variable $S_N$, evaluated at the event $S_N=n$; related uses of this idea appear in \cite{holst1981conditional,janson2001moment,klein2019conditional} and the references therein. 

Define random variables
    \begin{align}
        \label{eq:XY_definitions}
    X := \tilde{O}_1,\qquad \text{and} \quad Y := w_{\tilde{O}_1}, 
\end{align}
and their means and variances $\mu_X := \ex{X}$, $\sigma_X^2=\Var{X}$, $\mu_Y := \ex{Y}$, and $\sigma_Y^2:=\Var{Y}$. Define sequences
\begin{align}
    a_n^2 := \Var{S_N} = N\sigma_X^2,\quad \text{ and } \quad b_n^2 := \Var{\tilde T_N(w)} = N\sigma_Y^2.
    \label{eq:sum_variance_sequences}
\end{align}

Because $S_N$ takes values on the lattice $\mathbb{Z}$, we have the pointwise identity
\begin{align}
    \label{eq:inversion-identity}
    \one_{\{S_N=n\}} = \frac{1}{2\pi}\int_{-\pi}^{\pi} e^{i u (S_N - n)}\, \dd u.
\end{align}
Taking expectations in \eqref{eq:inversion-identity}, interchanging expectation and integral (justified by Fubini's theorem, as the integrand is bounded by $1$ on a finite interval), and using $\mu_X=n/N$ together with the independence of $\tilde{O}_i$ over $i=1,\ldots,N$, we obtain
\[
\Prp{S_N = n} = \frac{1}{2\pi}\int_{-\pi}^{\pi} \ex{e^{i u (S_N - n)}} \dd u = \frac{1}{2\pi}\int_{-\pi}^{\pi} \left(\ex{e^{i u (X - \mu_X)}} \right)^N \dd u.
\]
Likewise, multiplying \eqref{eq:inversion-identity} by
$\exp\{it(\tilde T_N(w)-\ex{\tilde T_N(w)})/b_n\}$ before taking expectations, we get
\begin{align}
\ex{
\exp\left\{it\frac{\tilde T_N(w)-\ex{\tilde T_N(w)}}{b_n}\right\}
\one_{\{S_N=n\}}}
 & = \frac{1}{2\pi}\int_{-\pi}^{\pi}
\left( \ex{ e^{it \frac{Y-\mu_Y}{b_n}+iu(X-\mu_X)}} \right)^N
\dd u \nonumber \\
& = \frac{1}{2\pi}\int_{-\pi}^{\pi}  \left( \Psi_n(t,a_n u) \right)^N
\dd u,
\label{eq:tilted-llt-integral}
\end{align}
where we defined the characteristic function 
\[
\Psi_n(t,v):= \ex{ e^{\frac{i}{\sqrt{N}} \left(t \frac{Y-\mu_Y}{\sigma_Y}+v \frac{X-\mu_X}{\sigma_X}\right)}}.
\]
Notice that $\left(\Psi_n(t,v)\right)^N$ is the characteristic function of the standardized version of $(\tilde{T}_N(w),S_N)$. The following lemma provides an estimate of the inverse Fourier integral for this characteristic function at the local point $S_N=n$. This lemma is a conditional version of a lattice local central limit theorem for triangular arrays; see, for
example, \cite[Ch.~VII]{petrov1975sums} or
\cite[Ch.~19]{bhattacharya1976normal} and
\cite{holst1981conditional,janson2001moment,klein2019conditional} for related conditional limit theorems. The proof is provided in Appendix~\ref{sec:proof-of-lem-tilted-llt}.
\begin{lem}
    \label{lem:tilted-llt}
    For a fixed $t \in \reals$, as $n \to \infty$,
\begin{align}
\label{eq:tilted-llt}
\frac{1}{2\pi}\int_{-\pi a_n}^{\pi a_n}\Psi_n(t,v)^N\,\dd v
\longrightarrow
\frac{1}{\sqrt{2\pi}}e^{-t^2/2}.
\end{align}
\end{lem}

Changing variables in \eqref{eq:tilted-llt-integral} to $v=a_n\ u$ and applying Lemma~\ref{lem:tilted-llt}, yields
\begin{align}
& a_n\,
\ex{
\exp\left\{it\frac{\tilde T_N(w)-\ex{\tilde T_N(w)}}{\sqrt{\Var{\tilde T_N(w)}}}\right\}
\one_{\{S_N=n\}}} \nonumber \\
& \qquad \qquad =
\frac1{2\pi}\int_{-\pi a_n}^{\pi a_n}\left(\Psi_n(t,v) \right)^N\,\dd v \to \frac1{\sqrt{2\pi}}e^{-t^2/2}.
\label{eq:tilted-llt-integral-limit}
\end{align}
In particular, with $t=0$ we get
\begin{align}
a_n\,\Prp{S_N=n}
\longrightarrow
\frac1{\sqrt{2\pi}}.
\label{eq:tilted-llt-integral-zero}
\end{align}
The definition of the conditional expectation for the indicator random variables, and applying \eqref{eq:tilted-llt-integral-limit} and \eqref{eq:tilted-llt-integral-zero}, give
\begin{align*}
    \ex{
\exp\left\{it\frac{\tilde{T}_N(w)-\ex{\tilde{T}_N(w)}}
{\sqrt{\Var{\tilde T_N(w)}}}\right\} \mid S_N=n} & = \frac{\ex{
    \exp\left\{it\frac{\tilde T_N(w)-\ex{\tilde T_N(w)}}{b_n}\right\}
    \one_{\{S_N=n\}}}}{\Prp{S_N=n}} \\
    & = \frac{a_n \ex{
        \exp\left\{it\frac{\tilde T_N(w)-\ex{\tilde T_N(w)}}{b_n}\right\}
        \one_{\{S_N=n\}}}}{a_n \Prp{S_N=n}} \\
        & \longrightarrow e^{-t^2/2}.
\end{align*}
Since we obtained the standard normal characteristic function, Levy's continuity theorem (cf.\ \cite[Thm.~2.13]{van2000asymptotic}) implies
\[
\frac{U_N(w)-\ex{\tilde T_N(w)}}
{\sqrt{\Var{\tilde T_N(w)}}}
\overset{d}{\longrightarrow} \Ncal(0,1).
\]
Finally, convergence in distribution to a continuous distribution function is uniform by P\'olya's theorem (cf.\ \cite[Lem.~2.11]{van2000asymptotic}), thus
\[
\sup_{x\in\reals}
\left| \Prp{\left\{ \frac{U_N(w)-\ex{\tilde T_N(w)}}
{\sqrt{\Var{\tilde T_N(w)}}} \le x \right\}} -\Phi(x) \right|
\longrightarrow 0.
\]
\end{proof}

\section{Proof of technical lemmas}
\label{sec:proofs-app}

\subsection{Proof of Lemma~\ref{lem:tilted-llt}}
\label{sec:proof-of-lem-tilted-llt}
\begin{proof}
    We fix $t \in \reals$ and a constant $W \in \reals$. We divide the integration range $|v|\leq \pi a_n$ into three areas:
\begin{itemize}
    \item (I) $|v| < W$.
    \item (II) $W \leq |v| < \varepsilon_t a_n$ for some $\varepsilon_t > 0$ that will be determined later. 
    \item (III) $\varepsilon_t a_n \leq |v| \le \pi a_n$.
\end{itemize}
The proof will show that region (I) provides the Gaussian contribution, whereas regions (II) and (III) have negligible contributions as $W \to \infty$. This decomposition to regions is common in lattice local central limit theorems, and can be found, for example, in \cite{petrov1975sums,holst1979unified,janson2001moment,klein2019conditional}.

We first establish the asymptotic scale of the sequence $a_n$. Let $\Lambda := nQ_1$. We have $\ex{\Lambda}= \mu_X = \frac{n}{N}$, and by (ii), $0 \le \Lambda \le C$. In particular, it follows that 
\[
    \Var{\Lambda} \leq \ex{\Lambda^2} \leq C \ex{\Lambda} = C\,n/N.
\]
By the law of total variance,
\[
\sigma_X^2=\ex{\Lambda}+\Var{\Lambda}, 
\]
which leads to
\[
\frac{n}{N} = \ex{\Lambda} \le \sigma_X^2 = \ex{\Lambda}+\Var{\Lambda} \le (1+C)\frac {n}{N}.
\]
Consequently 
\begin{align}
    \label{eq:a_n-bound}
n \le a_n^2 \le (1+C)n,
\end{align}
so $a_n\asymp \sqrt{n}$. 

We also need the following lemma, which provides a third moment bound to $X$, analogously to the bound for $Y$ provided by assumption (iii) of Theorem~\ref{thm:CLT}. 
\begin{lem}
    \label{lem:X-bound}
Under the setup of \eqref{eq:XY_definitions} and assumption (ii) of Theorem~\ref{thm:CLT},
\begin{align}
        \label{eq:X-bound}
    \frac{\ex{|X-\mu_X|^3}}{\sqrt{N} \sigma_X^3} = O(n^{-1/2}) \longrightarrow 0.
    \end{align}
\end{lem}
The proof of Lemma~\ref{lem:X-bound} is in Appendix~\ref{sec:proof-of-lem-X-bound}.
    
For fixed $(t,v)\in\reals^2$, define 
\[
    Z_{n}(t,v) := t\frac{Y-\mu_Y}{\sigma_Y}+v\frac{X-\mu_X}{\sigma_X}.
\]
We have, 
\[
\Psi_n(t,v)=\ex{e^{i Z_n(t,v)/\sqrt N}},
\quad \ex{Z_n(t,v)}=0, \text{~~and~~}
\Var{Z_n(t,v)} =: t^2+v^2+2\rho_n tv,
\]
where $\rho_n := \Corr{X,Y} = o(1)$ by assumption (i).

\subsubsection*{Region (I)}
A Taylor expansion of the complex exponent gives
\begin{align}
    \label{eq:taylor-expansion}
\left|e^{ix}-1-ix+\frac{x^2}{2}\right| \le A|x|^3,
\end{align}
for some constant $A<\infty$. Thus, uniformly for $|v|\le W$,
\[
\Psi_n(t,v)
=1-\frac{t^2+v^2+2\rho_n tv}{2N}
+O\left(\frac{\ex{|Z_n(t,v)|^3}}{N^{3/2}}\right).
\]
Likewise, for $|v|\le W$, there is a constant $B_{t,W}<\infty$ such that 
\[
\ex{|Z_n(t,v)|^3}
\le
B_{t,W}\left[
\ex{\left|\frac{Y-\mu_Y}{\sigma_Y}\right|^3}
+\ex{\left|\frac{X-\mu_X}{\sigma_X}\right|^3}
\right].
\]
By (iii) and Lemma~\ref{lem:X-bound}, we get
\[
    \sup_{|v|\le W} \frac{1}{\sqrt{N}}\ex{|Z_n(t,v)|^3} \longrightarrow 0,
\]
hence
\[
\sup_{|v|\le W} \left| N(\Psi_n(t,v)-1) + \frac{1}{2} \left(t^2+v^2+2\rho_n tv\right) \right|
\longrightarrow 0.
\]
Equivalently, uniformly for $|v|\le W$,
\begin{align}
    \label{eq:tilted-llt-tail-estimate-uniform}
N(\Psi_n(t,v)-1) = -\frac{1}{2}\left(t^2+v^2+2\rho_n tv\right)+o(1).
\end{align}
Since $t$ and $W$ are fixed and $\rho_n=o(1)$ by (i), the right-hand side of 
\eqref{eq:tilted-llt-tail-estimate-uniform} is uniformly bounded. As \eqref{eq:tilted-llt-tail-estimate-uniform} implies $\sup_{|v|\le W}|\Psi_n(t,v)-1|=O(1/N)$, a Taylor expansion yields
\begin{align}
    \label{eq:tilted-llt-tail-log}
N\log\Psi_n(t,v)
=N\{\Psi_n(t,v)-1\}+o(1)
\end{align}
uniformly for $|v|\le W$. By \eqref{eq:tilted-llt-tail-estimate-uniform} and \eqref{eq:tilted-llt-tail-log}, we conclude that
\begin{align}
    \label{eq:tilted-llt-tail-estimate-uniform-I}
\lim_{n \to \infty} \Psi_n(t,v)^N
= 
\exp\left\{ \lim_{n \to \infty} N \left( \Psi_n(t,v)-1\right)\right\}
= \exp\left\{-\frac{1}{2}(t^2+v^2)\right\}
\end{align}
uniformly for $|v|\le W$. Therefore, for every fixed $W<\infty$,
\begin{align}
    \label{eq:tilted-llt-compact-limit}
\int_{|v|\le W} \Psi_n(t,v)^N\,\dd v
\longrightarrow
\int_{|v|\le W} \exp\left\{-\frac{1}{2}(t^2+v^2)\right\}\,\dd v.
\end{align}

\subsubsection*{Region (II)}
On the region $W \le |v| \le \varepsilon_t a_n$, set
\[
V_{n,v}:= \Var{Z_n(t,v)} = t^2+v^2+2\rho_n tv.
\]
Since $\rho_n=o(1)$, for all large $n$ we have $V_{n,v} \ge \frac{1}{2}(t^2+v^2)$. We also have, 
$v^2\le \varepsilon_t^2a_n^2$, $a_n^2/N=\sigma_X^2 \le (1+C)n/N\le C(1+C)$, and $t^2/N\to0$. Thus, by decreasing $\varepsilon_t$ if necessary, we may also assume that $V_{n,v}/(2N)\le 1$ throughout this region for all large $n$. 

Let
\[
L_{Y,n}:=\frac{\ex{ |(Y-\mu_Y)/\sigma_Y|^3}}{\sqrt{N}},
\qquad
L_{X,n}:=\frac{\ex{ |(X-\mu_X)/\sigma_X|^3}}{\sqrt{N}}.
\]
By (iii) $L_{Y,n} \to 0$ and by \eqref{eq:X-bound} $L_{X,n}=O(n^{-1/2})$. From these and the lower bound on $V_{n,v}$ above, there is a constant  $B<\infty$ such that
\begin{equation}
\label{eq:third-moment-ratio-II}
\frac{\ex{ |Z_n(t,v)|^3}}{\sqrt N\,V_{n,v}}
\le B\left[ \frac{|t|^3L_{Y,n}}{V_{n,v}} +\frac{|v|^3L_{X,n}}{V_{n,v}} \right].
\end{equation}
For $W\ge 1$, uniformly over $W \le |v|\le \varepsilon_t a_n$, the first term in \eqref{eq:third-moment-ratio-II} is
$o(1)$ as $n\to\infty$, while the second term is at most
\[
B L_{X,n}|v| \le B n^{-1/2}\varepsilon_t a_n \le B'\varepsilon_t.
\]
Choose $\varepsilon_t$ small enough that the last bound is less than
$1/(8A)$, and then take $n$ large enough that the first term is also less
than $1/(8A)$. We now apply \eqref{eq:taylor-expansion} with
$x=Z_n(t,v)/\sqrt{N}$. Since $\ex{Z_n(t,v)}=0$ and
$\ex{Z_n(t,v)^2}=V_{n,v}$, taking expectations gives
\begin{align}
    \label{eq:psi-taylor-II}
\Psi_n(t,v)
=
1-\frac{V_{n,v}}{2N}+r_{n,v},
\end{align}
where
\begin{align}
    \label{eq:psi-taylor-II-remainder}
|r_{n,v}|
= \left|\ex{e^{iZ_n(t,v)/\sqrt{N}}-1-\frac{iZ_n(t,v)}{\sqrt{N}}+\frac{Z_n(t,v)^2}{2N}}\right|
\le \frac{A\,\ex{|Z_n(t,v)|^3}}{N^{3/2}}.
\end{align}
Combining \eqref{eq:psi-taylor-II-remainder} with \eqref{eq:third-moment-ratio-II}, we obtain, uniformly on $W \le |v|\le \varepsilon_t a_n$,
\[
|r_{n,v}|
= \frac{A}{N}\cdot
\frac{\ex{|Z_n(t,v)|^3}}{\sqrt{N}\,V_{n,v}}\cdot V_{n,v}
< \frac{A}{N}\cdot\frac{1}{4A}\cdot V_{n,v}
= \frac{V_{n,v}}{4N},
\]
where the strict inequality uses the choice of $\varepsilon_t$ and large $n$
above. Since $V_{n,v}/(2N)\le 1$, this and \eqref{eq:psi-taylor-II} imply
\[
|\Psi_n(t,v)| \le 1-\frac{V_{n,v}}{4N}.
\]
Using the inequality $(1-x/N)^N \le e^{-x}$, we get, for all sufficiently large $n$,
\[
|\Psi_n(t,v)|^N \le \exp\{-V_{n,v}/4\} \le \exp\{-c(t^2+v^2)\}
\]
with $c>0$ independent of $n$, $W$, and $v$. Consequently, for every $W\ge1$,
\[
\limsup_{n\to\infty}
\int_{W \le |v|\le \varepsilon_t a_n}|\Psi_n(t,v)|^N\,\dd v
\le
\int_{|v|\ge W} e^{-c(t^2+v^2)}\,\dd v.
\]
The right-hand side tends to zero as $W\to\infty$. Hence
\begin{align}
    \label{eq:tilted-llt-tail-estimate-limsup}
\lim_{W\to\infty}\limsup_{n\to\infty}
\int_{W \le |v|\le \varepsilon_t a_n}|\Psi_n(t,v)|^N\,\dd v=0.
\end{align}

\subsubsection*{Region (III)}
Write $u=v/a_n$. Then $\varepsilon_t\le |u|\le \pi$. Since
$\ex{e^{iuX} \mid \Lambda} = \exp\{\Lambda(e^{iu}-1)\}$,
\[
|\ex{ e^{iuX}}|
\le
\ex{\left| e^{\Lambda(e^{iu}-1)} \right|}
=
\ex{ e^{-\Lambda(1-\cos u)}}.
\]
For fixed $u$, the function $g(\lambda) := 1-e^{-\lambda(1-\cos u)}$ is concave on $[0,C]$ with $g(0)=0$, hence $g(\lambda) \ge (\lambda/C)\, g(C)$ there. Since $1-\cos u$ is even and increasing on $[0,\pi]$, we have $1-\cos u \ge 1-\cos\varepsilon_t$ for $\varepsilon_t \le |u| \le \pi$, whence
\[
1-e^{-\Lambda(1-\cos u)}
\ge \frac{\Lambda}{C}\left(1-e^{-C(1-\cos u)}\right)
\ge \kappa(\varepsilon_t) \Lambda,
\qquad
\kappa(\varepsilon_t)
:=
\frac{1-\exp\{-C(1-\cos\varepsilon_t)\}}{C}>0.
\]
Hence, 
\begin{align}
    \label{eq:tilted-llt-RIII-X-bound}
|\ex{ e^{iuX}}|
\le 1-\kappa(\varepsilon_t) \ex{\Lambda}
=1-\kappa(\varepsilon_t)\frac nN.
\end{align}
For the $Y$ part, by the inequalities $|e^{i\theta}-1|\le|\theta|$ and $\ex{|Y-\mu_Y|}\le \sigma_Y$ (Jensen's) with $b_n=\sqrt{N}\sigma_Y$, we have
\begin{align}
    \label{eq:tilted-llt-RIII-Y-bound}
    \ex{\left|e^{it\frac{Y-\mu_Y}{b_n}}-1\right|} & 
    \le \ex{\left|t\frac{Y-\mu_Y}{b_n}\right|} \le \frac{|t| \sigma_Y}{b_n} = \frac{|t|}{\sqrt{N}}. 
\end{align}
By \eqref{eq:tilted-llt-RIII-X-bound} and \eqref{eq:tilted-llt-RIII-Y-bound}, we have
\begin{align*}
    |\Psi_n(t, a_n u)| & = \left| \ex{e^{iu(X-\mu_X)}} + \ex{e^{iu(X-\mu_X)}\left(e^{it \frac{Y-\mu_Y}{b_n}} - 1 \right)} \right| \\
    & \le \ex{\left|e^{iu(X-\mu_X)}\right|} + \ex{\left|e^{it \frac{Y-\mu_Y}{b_n}}-1\right|} \\
    & \le 1-\kappa(\varepsilon_t)\frac{n}{N} + \frac{|t|}{\sqrt N}.
\end{align*}
Thus, uniformly for
$\varepsilon_t a_n \le |v|\le \pi a_n$,
\[
|\Psi_n(t,v)|
\le 1-\frac{\kappa(\varepsilon_t)}{2}\frac{n}{N}
\]
for all large $n$. The base is nonnegative, since $\kappa(\varepsilon_t)\, n/N \le \kappa(\varepsilon_t)\, C = 1-e^{-C(1-\cos\varepsilon_t)} < 1$. Hence, using $1-x\le e^{-x}$ again,
\begin{align}
    \label{eq:tilted-llt-tail-estimate}
\int_{\varepsilon_t a_n\le |v|\le \pi a_n}
|\Psi_n(t,v)|^N\,\dd v
\le 2\pi a_n\exp\{-\kappa(\varepsilon_t) n/2\}
\longrightarrow 0.
\end{align}

\subsubsection*{Combining the three regions}
We now combine the three estimates. Fix $W\ge 1$. Since $a_n \asymp \sqrt{n}$ by \eqref{eq:a_n-bound}, for all sufficiently large $n$ we have $W \le \varepsilon_t a_n$. Hence
\begin{align}
\label{eq:tilted-llt-final-decomposition-1}
&\left|
\int_{-\pi a_n}^{\pi a_n}\Psi_n(t,v)^N\,\dd v
-
\int_{-\infty}^{\infty}e^{-(t^2+v^2)/2}\,\dd v
\right| \\
\label{eq:tilted-llt-final-decomposition-2}
&\le
\left|
\int_{|v|\le W}
\left[\Psi_n(t,v)^N-e^{-(t^2+v^2)/2}\right] \dd v
\right| \\
\label{eq:tilted-llt-final-decomposition-3}
&\quad+
\int_{W \le |v|\le \varepsilon_t a_n}|\Psi_n(t,v)|^N\,\dd v
+
\int_{\varepsilon_t a_n\le |v|\le \pi a_n}|\Psi_n(t,v)|^N\,\dd v
+
\int_{|v|>W}e^{-\frac{t^2+v^2}{2}}\,\dd v .
\end{align}
Taking $\limsup_{n\to\infty}$ in \eqref{eq:tilted-llt-final-decomposition-1}, the term \eqref{eq:tilted-llt-final-decomposition-2} vanishes by \eqref{eq:tilted-llt-compact-limit}, the first term in \eqref{eq:tilted-llt-final-decomposition-3} vanishes by \eqref{eq:tilted-llt-tail-estimate-limsup}, and the second term in \eqref{eq:tilted-llt-final-decomposition-3} vanishes by \eqref{eq:tilted-llt-tail-estimate}. 
It follows that as $W\to\infty$, terms in \eqref{eq:tilted-llt-final-decomposition-2} and \eqref{eq:tilted-llt-final-decomposition-3} vanish. Therefore
\begin{align}
\label{eq:tilted-llt-unnormalized-limit}
\int_{-\pi a_n}^{\pi a_n}\Psi_n(t,v)^N\,\dd v
\longrightarrow
\int_{-\infty}^{\infty}e^{-(t^2+v^2)/2}\,\dd v
=\sqrt{2\pi}\,e^{-t^2/2}.
\end{align}
Multiplying \eqref{eq:tilted-llt-unnormalized-limit} by $1/(2\pi)$ gives \eqref{eq:tilted-llt}.
\end{proof}

\subsection{Proof of Lemma~\ref{lem:X-bound}}
\label{sec:proof-of-lem-X-bound}
\begin{proof}
If $\Upsilon_{\lambda}\sim\Pois(\lambda)$ for $0\le\lambda\le C$, then 
$\ex{\Upsilon_{\lambda}^3}=\lambda^3+3\lambda^2+\lambda\le B_C\lambda$ with $B_C := C^2+3C+1<\infty$. 

With $\Lambda = nQ_1$ and $\mu_X = \ex{\Lambda}=n/N$, we have $X | \Lambda \sim \Pois(\Lambda)$ and $\Lambda \sim \pi_1$. Therefore,
\[
\ex{X^3} = \ex{\ex{X^3 | \Lambda}} = \ex{\Lambda^3 + 3\Lambda^2 + \Lambda} \leq B_C\ex{\Lambda}.
\]
Using the inequality $|x-\mu_X|^3\le 4(x^3+\mu_X^3)$ with $\mu_X\le C$ and  $\mu_X^3\le C^2\mu_X$, we conclude
\[
\ex{|X-\mu_X|^3}\le A_C\mu_X,
\]
for some constant $A_C<\infty$. This implies
\[
    \frac{\ex{|X-\mu_X|^3}}{\sqrt{N} \sigma_X^3} \leq \frac{A_C\mu_X}{\sqrt{N}\,\mu_X^{3/2}} = O(n^{-1/2}).
\]
\end{proof}

\subsection{Proof of Lemma~\ref{lem:conditional_limit_theorem}}
\label{sec:proof_conditional_limit_theorem}
\begin{proof}
Direct evaluation using the Poisson central moments shows that the mean and variance of $\tilde{T}(w^*)$ under $\tilde{H}_0 := \tilde{H}(\mathbf{q}_{\unif})$ are 
\begin{align}
    \label{eq:mu_0_sigma_0}
     \mu_0 &:= \ex{\tilde{T}(w^*)\mid \tilde{H}_0} = 0, \qquad
     \sigma_0^2 := 2 n^2 / N,
 \end{align}
 and under $\tilde{H}(\pi^*)$, using also the law of total variance, 
 \begin{align}
    \label{eq:mu_1_sigma_1}
     \mu_1 &:= \ex{\tilde{T}(w^*)\mid \tilde{H}(\pi^*)} = n^2 \epsilon^2 N^{1-2/p}, \qquad \sigma_1^2 := \sigma_0^2(1+o(1)).
 \end{align}
 Notice that $\mu_1/\sigma_0 = u_n$.
 
 We now verify that both unconditioned laws $\tilde{H}_0$ and $\tilde{H}(\pi^*)$ satisfy the conditions on $\tilde{O}_1,\ldots,\tilde{O}_N$ in Theorem~\ref{thm:CLT}. 
For condition (i), set $\lambda_0:=n/N$ and $\Delta_n:=n\epsilon N^{-1/p}$. If $X\sim\Pois(\lambda)$, then a direct calculation from the first three Poisson moments gives
\[
\ex{w_X^*}=(\lambda-\lambda_0)^2,\qquad
\Cov{X,w_X^*}=2\lambda(\lambda-\lambda_0).
\]
For $\tilde{H}_0$, $\lambda=\lambda_0$, so the covariance, and hence the correlation, is zero. For $\tilde{H}(\pi^*)$, let $\Lambda=nQ_1=\lambda_0\pm\Delta_n$. The law of total covariance and the symmetry of $\pi^*$ give
\[
\Cov{\tilde{O}_1,w^*_{\tilde{O}_1}}
:=\ex{2\Lambda(\Lambda-\lambda_0)}
    +\Cov{\Lambda,(\Lambda-\lambda_0)^2}
:=2\Delta_n^2.
\]
Also $\Var{\tilde{O}_1}=\lambda_0+\Delta_n^2=\lambda_0(1+o(1))$, since
\[
\frac{\Delta_n^2}{\lambda_0}=\epsilon^2 nN^{1-2/p}=O(N^{-1/2})=o(1),
\]
and \eqref{eq:mu_1_sigma_1} gives $\Var{w^*_{\tilde{O}_1}}=2\lambda_0^2(1+o(1))$. Therefore
\[
\Corr{\tilde{O}_1,w^*_{\tilde{O}_1}}
=O\left(\frac{\Delta_n^2}{\lambda_0^{3/2}}\right)
=O(u_n/\sqrt{n})=o(1),
\]
which proves condition (i) under both laws.
For condition (ii), under $\tilde{H}_0$ the prior is $\delta_{1/N}$, while under $\tilde{H}(\pi^*)$ it is the two-point distribution in \eqref{eq:least_favorable_prior}. Thus $\ex{Q_1}=1/N$ and the support is finite in both cases. The support bound required in (ii) is $n \epsilon N^{-1/p} =O(1)$, which holds since $u_n \to u^*$ implies $\Delta_n^4 = O(n^2 N^{-3}) = O(N^{-1})$. 

It remains to verify condition (iii). First, conditioned on any $Q_1$,
\[
\Var{w^*_{\tilde{O}_1}} \geq \ex{\Var{w^*_{\tilde{O}_1}\mid Q_1}}
=\ex{2(nQ_1)^2}>0,
\]
by the law of total variance. This implies that $\Var{w^*_{\tilde{O}_1}}>0$ under both $\tilde{H}_0$ and $\tilde{H}(\pi^*)$. For the Lyapunov condition, note that in the bounded-rate regime $\lambda_0$ and $\Lambda$ are uniformly bounded, and $\Delta_n/\lambda_0=o(1)$. Since $w_m^*$ is quadratic in $m$, the Poisson moment formulas up to order six imply, uniformly for $\lambda\in\{\lambda_0,\lambda_0-\Delta_n,\lambda_0+\Delta_n\}$,
\[
\ex{\left|w_X^*-\ex{w_X^*}\right|^3}=O(\lambda_0^2),\qquad X\sim\Pois(\lambda).
\]
The same bound holds after mixing over $\pi^*$. Combining this with
$\Var{w^*_{\tilde{O}_1}}=2\lambda_0^2(1+o(1))$ under both $\tilde{H}_0$ and $\tilde{H}(\pi^*)$ gives
\[
\frac{\ex{\left|w^*_{\tilde{O}_1}
-\ex{w^*_{\tilde{O}_1}}\right|^3}}
{\sqrt{N}\Var{w^*_{\tilde{O}_1}}^{3/2}}
:=O\left(\frac{1}{\sqrt{N}\lambda_0}\right)
:=O\left(\frac{\sqrt{N}}{n}\right)=o(1),
\]
where the last step uses $N=o(n^2)$. This proves condition (iii).
By Theorem~\ref{thm:CLT}, $u_n \to u^*$, \eqref{eq:mu_0_sigma_0}, and \eqref{eq:mu_1_sigma_1}, we have the following convergence in distribution:
\begin{align}
    \label{eq:CLT_result}
    \frac{U(w^*)}{\sigma_0}  \to \begin{cases}
        \Ncal(0,1), & \tilde{H}_0(\mathbf{q}_{\unif} \mid \Scal_n), \\
        \Ncal(u^*,1), & \tilde{H}(\pi^* \mid \Scal_n).
    \end{cases}
\end{align}
Namely, we are facing the problem of testing in a simple Gaussian shift experiment. The minimal risk in this case is 
\[
\tilde{\rho}(\pi^*; \psi_{w^*,\tau^*} \mid \Scal_n) = 2\Phi(-u_n(1+o(1))/2) = 2\Phi(-u^*/2) + o(1),
\]
and the threshold attaining it is $\tau^* = \mu_1 / 2$.
\end{proof}

\subsection{Proof of Lemma~\ref{lem:Bayes_equivalence}}
\label{sec:proof_bayes_equiv}
\begin{proof}
Denote by $P_\pi$ the probability law $\mathbf{Q} \sim \pi$ where $\mathbf{Q} = (Q_1,\ldots,Q_N)$ with $Q_i \sim \pi_i$ independently for $i=1,\ldots,N$. By the definition of $\Pi_N$, each $\pi_i$ is supported on $[0,1]$, hence the variance of $|Q_i - 1/N|^p$ exists and is uniformly bounded in $i$.

Because $\pi = \prod_{i=1}^N \pi_i$, the law of large numbers implies
\begin{align*}
& \liminf_{N \to \infty} \frac{1}{N}\sum_{i=1}^N \epsilon^{-p}\left|Q_i - 1/N\right|^p \\
&= \liminf_{N \to \infty} \frac{\epsilon^{-p}}{N}\sum_{i=1}^N \ex{\left|Q_i - 1/N\right|^p} \geq 1
\end{align*}
$P_{\pi}$ almost surely, where we used that $\pi \in \Pi_N$ in the last inequality. It follows that for some sequence $\epsilon_N$ satisfying $\epsilon_N/\epsilon \to 1$ and $\epsilon_N/\epsilon \geq 1$, we have 
\begin{align}
    \label{eq:pi_Q_in_tilde_A_N}
    P_{\pi}\left[\mathbf{Q} \in \tilde{A}_N(\epsilon_N,p)\right] = \pi(\tilde{A}_N(\epsilon_N,p)) \to 1.
\end{align}
In particular, \eqref{eq:pi_Q_in_tilde_A_N} implies that $P_{\pi}{\sum_{i=1}^N Q_i = 0} \to 0$, so 
\[
P_{\pi}[\Scal_n] = \ex{\frac{e^{-n\sum_{i=1}^N Q_i}(n\sum_{i=1}^N Q_i)^n}{n!}} > 0.
\]
Therefore, we may define the conditional measure $\bar{\pi}(A) = \pi(A \cap \tilde{A}_N(\epsilon_N,p) \cap \Scal_n )/\pi(\tilde{A}_N(\epsilon_N,p)\cap \Scal_n)$. We have $\bar{\pi}(\tilde{A}_N(\epsilon_N,p))=1$, and hence by considering the point mass prior at any $\mathbf{q} \in \tilde{A}_N(\epsilon_N,p)$, for any test $\psi$ we have
\[
\Prp{\psi = 0 \mid \tilde{H}(\bar{\pi} | \Scal_n )} \leq \tilde{R}^*(\tilde{A}_N(\epsilon,p) \mid \Scal_n). 
\]
It follows that
\[
\tilde{\rho}^*(\bar{\pi} \mid \Scal_n) \leq \tilde{R}^*(\tilde{A}_N(\epsilon_N,p) \mid \Scal_n). 
\]

On the other hand, the total variation distance obeys $\mathrm{TV}(P_\pi,P_{\bar{\pi}})\to 0$ by \eqref{eq:pi_Q_in_tilde_A_N}. It follows that
\begin{align*}
\tilde{\rho}^*(\pi \mid \Scal_n) & = \tilde{\rho}^*(\bar{\pi} \mid \Scal_n)+o(1) \leq \tilde{R}^*(\tilde{A}_N(\epsilon_N,p) \mid \Scal_n)+o(1) \\
& = \tilde{R}^*(\tilde{A}_N(\epsilon,p) \mid \Scal_n)+o(1) = \tilde{R}^* + o(1),
\end{align*}
where the penultimate transition follows from continuity of $\tilde{R}^*(\tilde{A}_N(\epsilon,p))$ in $\epsilon$.
\end{proof}
\begin{remark}
The conditioning on $\Scal_n$ plays no essential role in the arguments above, which rely on the behavior of $\mathbf{Q}$ under the prior $\pi$ alone. Specifically, Lemma~\ref{lem:Bayes_equivalence} holds with conditioning on any other event with non-zero probability under $\Pr_\pi$. 
\end{remark}

\subsection{Proof of Lemma~\ref{lem:LR_equivalence}}
\label{sec:proof_LR_equivalence}
\begin{proof}
By Bayes' rule, it is enough to prove that
\[
    \Prp{\Scal_n \mid \tilde{H}(\pi^*)}/\Prp{ \Scal_n \mid \tilde{H}(\mathbf{q}_{\unif})} = 1+O(1/\sqrt{N}).
\]
The sum of Poisson random variables is a Poisson random variable, hence
\begin{align*}
\Prp{\Scal_n \mid \tilde{H}(\mathbf{q}_{\unif})} & = \Prp{\sum_{i=1}^N \tilde{O}_i = n \mid \tilde{H}(\mathbf{q}_{\unif})} = \frac{e^{-n}n^n}{n!}.
\end{align*}
When $\mathbf{Q}$ is drawn from $\pi^*$, the rate of $\sum_{i=1}^N \tilde{O}_i$ is the random variable $\Lambda := n + nN^{-1/p}\epsilon(2B-N)$ where $B$ indicates the number of upward perturbations of $1/N$ ('$+$' masses) in $Q$. Since $B \sim \Bin(N,1/2)$ (binomial), we get
\begin{align*}
    \Prp{\sum_{i=1}^N \tilde{O}_i = n \mid \tilde{H}(\pi^*)} & = \ex{ \frac{e^{-\Lambda}}{n!} \Lambda^n } \\
    & = \frac{e^{-n}}{n!}n^n \ex{ e^{-nN^{-1/p}\epsilon(2B-N)} \left(1 + N^{-1/p}\epsilon(2B-N)\right)^n }.
\end{align*}
From the inequality $(1+x/n)^n \leq e^x$, we get
\begin{align}
\label{eq:less_than_one}
\ex{ e^{-nN^{-1/p}\epsilon(2B-N)} \left(1 + N^{-1/p}\epsilon(2B-N)\right)^n } \leq 1.
\end{align}
To obtain a matching lower bound, set
\[
\Delta := N^{-1/p}\epsilon(2B-N),
\]
so that the expectation in \eqref{eq:less_than_one} is $\ex{\exp\big(n(\log(1+\Delta)-\Delta)\big)}$. On the event $\{|\Delta|\leq 1/2\}$, we use the bound
\[
\log(1+x) \geq x - x^2,\qquad |x|\leq 1/2,
\]
which implies
\[
\exp\big(n(\log(1+\Delta)-\Delta)\big) \geq \exp(-n\Delta^2).
\]
Therefore,
\begin{align}
\ex{ e^{-n\Delta}\left(1+\Delta\right)^n }
&= \ex{\exp\big(n(\log(1+\Delta)-\Delta)\big)} \nonumber \\
&\geq \ex{\exp(-n\Delta^2)\,\one_{\{|\Delta|\leq 1/2\}} }.
\label{eq:lower_than_one}
\end{align}
Since $B\sim\Bin(N,1/2)$, we have $2B-N = O_{p}(\sqrt{N})$, hence $\Delta = O_{p}(\epsilon N^{1/2-1/p})$ and, in particular, $\Prp{|\Delta|\leq 1/2}\to 1$ under our asymptotic regime. Moreover, on $\{|\Delta|\leq 1/2\}$,
\[
0\leq 1-\exp(-n\Delta^2) \leq n\Delta^2,
\]
so \eqref{eq:lower_than_one} yields
\[
\ex{ e^{-n\Delta}\left(1+\Delta\right)^n } \geq 1 - \ex{n\Delta^2} + o(1) = 1 - O\!\left(n\epsilon^2 N^{1-2/p}\right) + o(1).
\]
In particular, $\Prp{\Scal_n \mid \tilde{H}(\pi^*)}/\Prp{\Scal_n \mid H_0} = 1+O(1/\sqrt{N})$.
\end{proof}

\subsection{Taylor expansion for the likelihood ratio weights}
\label{sec:taylor_expansion}

Recall the kernel $g_{m,\lambda}$ from \eqref{eq:g_def} of Lemma~\ref{lem:unconditioned_LR_test},
\[
g_{m,\lambda}(s) = \frac{1}{2}e^{-\lambda s} \left( 1 + s \right)^{m} + \frac{1}{2}e^{\lambda s} \left( 1 - s \right)^{m}.
\]
We derive the Taylor expansion of $\log g_{m,\lambda}(s)$ as $s \to 0$.

\begin{lem}
\label{lem:taylor_expansion}
For any $s \in \reals$ with $|s| < 1$ and $m=0,1,\ldots$, we have
\begin{align}
\label{eq:taylor_f}
    g_{m,\lambda}(s) & = 1 + \frac{s^2}{2}\left( (m-\lambda)^2 - m \right) + o\!\left(s^2(e^m + \lambda^2)\right).
\end{align}
Consequently,
\begin{align}
\label{eq:taylor_log_f}
\log g_{m,\lambda}(s) = \frac{s^2}{2}\left( (m-\lambda)^2 - m \right)+o\!\left(s^2 (e^m + \lambda^2)\right).
\end{align}
\end{lem}

\begin{proof}
By the binomial expansion and exponential Taylor series, we have
\begin{align*}
    2 g_{m,\lambda}(s) &= e^{-\lambda s}(1+s)^m + e^{\lambda s}(1-s)^m\\
    &= \left(1 - \lambda s + \lambda^2 s^2/2 + o(s^2 \lambda^2)\right) \left(1 + m s + \binom{m}{2}s^2 + o(s^2 e^m)\right) \\
    &\quad + \left(1 + \lambda s + \lambda^2 s^2/2 + o(s^2 \lambda^2)\right) \left(1 - m s + \binom{m}{2}s^2 + o(s^2 e^m)\right).
\end{align*}
By Stirling's approximation, the maximal binomial coefficient $\binom{m}{k}$ is $O(2^m/\sqrt{m})$, which is $o(e^{m})$.
Collecting terms of order $s^2$ and lower, the linear terms cancel and we obtain
\begin{align*}
2 g_{m,\lambda}(s) &= 2 + s^2 \left(m(m-1) - 2\lambda m + \lambda^2 \right) + o(s^2(e^m + \lambda^2))\\
&= 2 + s^2 \left((m-\lambda)^2 - m \right) + o(s^2(e^m + \lambda^2)).
\end{align*}
Dividing by 2 yields \eqref{eq:taylor_f}.

For \eqref{eq:taylor_log_f}, we use the standard expansion $\log(1+z) = z + O(z^2)$ for $|z| \ll 1$. Since $g_{m,\lambda}(s) - 1 = O(s^2(\lambda^2 + e^m))$ as $s \to 0$, we have
\[
\log g_{m,\lambda}(s) = \left(g_{m,\lambda}(s) - 1\right) + O\!\left((g_{m,\lambda}(s)-1)^2\right) = \frac{s^2}{2}\left((m-\lambda)^2-m\right) + o(s^2(e^m + \lambda^2)),
\]
which completes the proof.
\end{proof}

\bibliographystyle{plainnat}
\bibliography{multinomials}

\end{document}